\documentclass[12pt]{article}

\usepackage{amsfonts}
\usepackage{amsmath}
\usepackage{amssymb}
\usepackage{amsthm}
\usepackage{setspace}
\usepackage{fullpage}
\usepackage{graphicx}
\usepackage{url}

\usepackage[numbers]{natbib}

\usepackage{verbatim}
\usepackage{graphicx}
\newtheorem{thm}{Theorem}[section]

\newtheorem{lem}[thm]{Lemma}

\newtheorem{cor}[thm]{Corollary}

\newtheorem{rem}[thm]{Remark}

\theoremstyle{remark}

\theoremstyle{definition}
\newtheorem{defn}[thm]{Definition}

\begin{document}

\author{Jacob McNamara}

\title{
A Bound on the Norm of Shortest Vectors in Lattices Arising from CM Number Fields
}

\date{}

\maketitle

\thispagestyle{empty}

\begin{abstract}

This paper partially addresses the problem of characterizing the lengths of vectors in a family of Euclidean lattices that arise from any CM number field $F$. We define a modified quadratic form on these lattices, the \emph{weighted norm}, that contains the standard field trace as a special case. Using this modified quadratic form, we obtain a bound on the field norm of any vector that has a minimal length in any of these lattices, in terms of a basis for $\mathcal{O}_F^{\times}$, the group of units of the ring of integers of the field $F$. For any CM number field $F$, we prove that there exists a finite set of elements of $F$ which allows one to find the set of minimal vectors in every principal ideal of the ring of integers of $F$. We interpret our result in terms of the asymptotic behavior of a Hilbert modular form, and consider some of the computational implications of our theorem. Additionally, we show how our result can be applied to the specific Craig's Difference Lattice problem, which asks us to find the minimal vectors in lattices arising from cyclotomic number fields.

\end{abstract}

\newpage

\section{Introduction}

In this paper, we consider a family of lattices that are related to CM number fields, central objects of study in number theory. By \emph{lattice}, we mean a discrete subgroup of $\mathbb{R}^n$ equipped with a quadratic form. A \emph{CM number field} is a number field for which complex conjugation can be defined consistently. It is this complex conjugation operation which allows us to generate a lattice from any CM number field, as described in Section \ref{CM}. Though lattices can be studied on their own as purely geometrical objects, there is a long history of lattices playing a role in the development of number theory. While many classical problems in number theory have been approached and solved using elementary methods, adopting the more sophisticated perspective based around lattices, and their associated modular forms, allows us to see how these classical problems fit into the larger and deeper abstract theory. This abstract perspective drives modern number theory research by suggesting new questions that seek to generalize results from elementary number theory. This paper considers, through the use of lattices, one of many possible generalizations of the classical $n$-squares problem. By looking at the history of the $n$-squares problem, we can see how the lattice approach naturally leads to the overarching question under investigation in this paper, namely, to characterize the lengths of vectors in lattices which arise from CM number fields.

The $n$-squares problem asks: given an integer
$m$, how many ways can $m$ be represented as a sum of $n$ squares,

\begin{equation}
\label{nsquares}
m = x_1^2 + x_2^2 + \dots + x_n^2,
\end{equation}

\noindent where the $x_k$ are integers? This problem is a collection of specific cases which have been studied by many mathematicians. For example, Lagrange's Four Squares Theorem addresses the $n$-squares problem for the case $n=4$ in the affirmative, by proving that any integer can be represented as a sum of four squares in at least one way. While Lagrange gave an elementary proof of the Four Squares Theorem in 1770 \cite{4squares}, a more sophisticated approach to the four squares problem involves the Hurwitz Integers (see \cite{Hurwitz}), a lattice in $\mathbb{R}^4$, which is the analog of the integers in the quaternions, a four-dimensional normed division algebra.

Similarly, while the two squares problem was originally solved with an elementary approach, it may also be solved with the use of the Gaussian Integers $\mathbb{Z}[i]$, a lattice in $\mathbb{C}$, as outlined in \cite{Gaussian}. The solution to the two squares problem is slightly more complicated than the Four Squares Theorem, in that an odd prime $p$ can be represented as a sum of two squares if and only if $p \equiv 1\ \text{mod}\ 4$. The general solution for two squares can be built up from the prime cases using the identity,

\begin{equation}
\label{normidentity}
(a^2 + b^2)(c^2 + d^2) = (ac-bd)^2 + (ad+bc)^2,
\end{equation}

\noindent which implies that the product of two numbers which are representable as a sum of two squares is itself a sum of two squares. This identity comes from the fact that if $z = a+bi$ and $w = c+di$ are two Gaussian Integers, the left hand side of Equation (\ref{normidentity}) is equal to $(z\bar{z}) (w\bar{w})$, and the right hand side is $(zw) (\bar{z}\bar{w})$. If we define the norm $N(z)$ of a Gaussian Integer by $N(z)=z\bar{z}$, then this is simply the fact that $N(zw)=N(z)N(w)$. This norm is the norm which makes $\mathbb{Z}[i]$ a lattice. From this perspective, the two squares problem asks us to find a vector $z \in \mathbb{Z}[i]$ with $N(z) = m$ for a given $m$.

The three squares problem is more complicated than either the four squares or two squares cases. Its solution requires the class numbers of purely imaginary quadratic fields \cite{Threesquares}, that is, quadratic number fields which can not be embedded into the real numbers. These purely imaginary quadratic fields are the most basic examples of CM number fields, and it is the fact that a CM number field gives rise to a natural quadratic form that allows us to solve the three squares problem.

These three cases indicate that the general $n$-squares problem, finding $\{x_1, x_2, \dots, x_n\}$ such that Equation (\ref{nsquares}) holds, is best described in terms of the lengths of vectors in lattices, and specifically in terms of lattices that arise from particular CM number fields. We consider the more general problem of characterizing the lengths of vectors in the lattices arising from any CM number field. Each CM number field and lattice can be associated with a Diophantine equation, such as the Equation (\ref{nsquares}) for the $n$-squares problem; however, by considering all CM number fields together, we can hope to prove statements which hold in a more general setting than if we simply considered each number field and Diophantine equation individually.

Our main result is an explicit bound, described in Theorem \ref{bound} and Corollary \ref{principal}, on the norm of the shortest vectors in the lattices arising from any CM number field, as well as the family of related lattices that arise from principal ideals of the field's ring of integers. The bound is formulated in terms of a basis for the group of units of the algebraic integers of the field. The key tool which allows us to prove our main theorem is the \emph{weighted norm}, an extension of the standard quadratic form on a number field. The weighted norm is useful in our proof because it depends on a set of weights, which can be adjusted to define different quadratic forms.

In addition to establishing our main bound, we prove Corollary \ref{finite}, which states that there exists a set of elements of the number field from which the minimal vectors in any lattice associated to any principal ideal of the number field can be determined. Consequently, our result can be applied so as to computationally simplify the problem of explicitly finding the shortest vectors in any particular lattice from a CM number field. This paper will highlight this use of our result by examining the Craig's Difference Lattice Problem, the problem of finding the shortest vectors in lattices arising from cyclotomic fields.

This computational significance of our result may help shed light on the well-known Shortest Vector Problem from computer science. This problem asks for an algorithm which, given a lattice, will return the shortest vectors in the lattice. It is known that, for general lattices, the problem is NP-hard \cite{NPhard}, and it is even NP-hard to approximate to within a factor of $\sqrt{2}$ \cite{constanthard}. Future research could determine whether our result can be used to create a polynomial time algorithm for the Shortest Vector Problem for these number theoretic lattices, to determine whether the additional information provided by the number fields from which they are derived is enough to reduce the computational complexity of the Shortest Vector Problem in this case. If so, it would indicate that these number theoretic lattices are significantly more well behaved than general lattices.

In Section \ref{motiv}, we provide some technical motivation for our
definitions, and describe the connection between our result and the theory of modular
forms. In Section \ref{prelims}, we establish definitions and notation
which is used later in the paper. In Section \ref{statement}, we state
our bound, which is proven in Section \ref{proof}. In Section
\ref{consequences}, some immediate consequences of our theorem are
explored, and in Section \ref{craig's} we apply our result to some
specific cases of the Craig's Difference Lattice problem.

\section{Motivation}
\label{motiv}

In this section, we lay the groundwork for understanding the definitions that we use in our proof. We first consider the Craig's Difference Lattice problem, an open problem which is a specific case of the general problem that this paper partially addresses, i.e., characterizing the lengths of vectors in lattices arising from CM number fields. We use the Craig's Lattices as a concrete example of the method of generating a lattice from a CM field. Also, we use this problem here, as an illustration, because we return to the Craig's Difference Lattice problem in Section \ref{craig's}, where we show how our main result can be applied to this problem. We then examine the connection between these lattices arising from CM fields and the more abstract theory of modular forms, and show how our definition of the weighted norm is a natural and useful extension of the standard quadratic form on a CM number field.

The Craig's Difference Lattice problem asks us to find an expression for
the shortest vectors in a certain infinite family of lattices which are
related to the cyclotomic number fields; for an introduction and additional details see Martinet \cite{Martinet}. A (Euclidean) \emph{lattice} is a discrete,
full-rank subgroup of $\mathbb{R}^n$ which inherits the inner product
from $n$ dimensional Euclidean space. Alternately, a rank-$n$ Euclidean lattice $\Lambda$ can be viewed as a free abelian group with an inner product, such that the extension of the inner product to $\Lambda \otimes_\mathbb{Z} \mathbb{R}$ is positive definite. For details on the theory of Euclidean lattices, see \cite{Serre}, \cite{Elkies}.

The Craig's Lattices are defined by taking iterated differences of vectors in earlier lattices in the
series. Let $T$ be the linear endomorphism of $\mathbb{R}^{n+1}$ which acts on the standard basis $\{e_0, e_1, \dots, e_n\}$ by $T(e_j) = e_{j+1}$, where the index $j$ is taken mod $n+1$. We define the Craig's Lattice $\mathbb{A}_n^{(r)}$ as the image of $\mathbb{Z}^{n+1}$ under the map $(1 - T)^r$. We see immediately that $\mathbb{A}_n^{(0)} = \mathbb{Z}^{n+1}$. Also, because $(1-T)$ is a $\mathbb{Z}$-module endomorphism with one eigenvalue 0, we see that $\mathbb{A}_n^{(r)}$ is a rank-$n$ sublattice of $\mathbb{Z}^{n+1}$ for $r>0$. Bachoc and Batut \cite{BacBat} have gotten
partial results on the shortest vectors in the lattice
$\mathbb{A}_n^{(r)}$, mainly in the case that $n$ is one less than a
prime. We will return to the Craig's Difference Lattice problem in Section \ref{craig's}, where we apply our main result to this problem to reduce the number of points that have to be
checked for any given $\mathbb{A}_n^{(r)}$ to a finite
number.

While our main theorem can be applied to simplify the Craig's
Difference Lattice problem, the Craig's Lattices are simply a special
case of a more general type of lattice arising from any \emph{CM number
  field}, i.e., any number field with a consistent complex
conjugation. The Craig's lattices occur as the rings of integers of the
cyclotomic fields and their principal ideals. The ring of integers of
the $p$th cyclotomic field for prime $p$ contains the prime ideal $\mathcal{P}$, generated by $(1 -
\zeta_p)$, where $\zeta_p$ is a primitive $p$th root of unity. The
Craig's lattice $\mathbb{A}_{p-1}^{(r)}$ is then isomorphic as a lattice
to $\mathcal{P}^r$ equipped with a suitable inner product.

This inner product can be defined for any CM number field, which is a
number field with a consistent complex conjugation. Let $F$ be a
CM field, and let $K$ be the real subfield of $F$ that is fixed under complex
conjugation, where we consider complex conjugation to be the nontrivial element of the Galois group $\text{Gal}(F/K) \cong \mathbb{Z}/2\mathbb{Z}$, as described formally in Section \ref{CM}. The field $F$
can, in general, be embedded into $\mathbb{C}$ in $2k$ different
ways for some $k\in \mathbb{N}$. Each conjugate pair of embeddings of $F$ into $\mathbb{C}$ gives
an embedding of $K$ into $\mathbb{R}$, for a total of $k$
embeddings. Denote these real embeddings of $K$ by $\sigma_j: K
\rightarrow \mathbb{R}$ for $1 \leq j \leq k$.

By composition with the norm from $F$ to $K$ which takes $\alpha$ to
$\alpha \bar{\alpha}$, each $\sigma_j$ defines a quadratic form on $F$ that takes
$\alpha$ to $\sigma_j(\alpha \bar{\alpha})$. We can
combine these norms into a single quadratic form by weighting each norm $\sigma_j$
with a positive real number $x_j$ and summing,

\begin{equation*}
\langle \alpha, \alpha \rangle_{x_1, x_2, \dots, x_k} = \sum\limits_{j=1}^k x_j \sigma_j(\alpha \bar{\alpha}).
\end{equation*}

\noindent We call this the \emph{weighted norm} on $F$ and  $\{
x_1,x_2,\dots,x_k \}$ the \emph{weights}. We extend this definition to $\langle \cdot, \cdot \rangle_{z_1,z_2,\dots,z_k}$ for complex $z_j$,
but this is a Euclidean norm only if all the $z_j$ are real and positive. If we restrict the
weighted norm to the ring of integers of $F$, denoted $\mathcal{O}_F$,
we get a family of lattices, one for each possible choice of weights $\{ x_1, \dots, x_k\}$, where each lattice is $\mathcal{O}_F$ as a group with the inner product $\langle \cdot, \cdot \rangle_{x_1, \dots, x_k}$. If $F$ is a cyclotomic field, we recover the zeroth Craig's
Lattices $\mathbb{A}_n^{(0)}$ by letting $x_j = 1$ for all $j$. Similarly, we can restrict
the weighted norm to a principal ideal of $\mathcal{O}_F$ to get more
lattices, including the lattices $\mathbb{A}_n^{(r)}$.

\begin{rem}
The standard quadratic form on a CM number field, $\text{Tr}_{K/\mathbb{Q}}(\alpha \bar{\alpha})$, is a special case of the weighted norm, namely when $x_j = 1$ for all $j$. However, the ability to vary the weights which define the weighted norm makes it more flexible than the field trace, in that it is able to describe a larger family of related lattices rather than one static lattice. We hope to show, in the following discussion of the relationship between these lattices and modular forms, that the weighted norm is a natural extension of the field trace, and that the weighted norm's flexibility allows us to encapsulate information about all the principal ideals of the ring of integers of $F$ in a single object. Further, its flexibility will be key in the proof of our main theorem.
\end{rem}

The weighted norm suggests a generalization of a fundamental invariant
of a lattice, the \emph{$\theta$ function}, defined for a lattice $\Lambda$ as,

\begin{equation*}                                                                                                                                                                 
\theta_\Lambda(z)=\sum\limits_{v \in \Lambda} e^{\pi i z \langle v, v \rangle/2},                                                                                                  
\end{equation*}

\noindent where $z$ is in the upper half plane $\mathcal{H} = \{ z \in \mathbb{C} \ |\ \Im (z) > 0 \}$ and $\langle \cdot, \cdot \rangle$ is the inner product of $\Lambda$. The $\theta$ function is a generating function for the lattice, in that, if we write $q =
e^{2 \pi i z}$, the coefficient of $q^{m/2}$ in $\theta_\Lambda(z)$ is the
number of points in $\Lambda$ with norm $m$. The $\theta$ function is a modular form for a
subgroup of the modular group $\text{SL}_2(\mathbb{Z})$, possibly of half integer weight, for large classes of lattices, as shown by
Elkies \cite{Elkies}.

However, if we
are to take seriously the fact that we are working in a field $F$, not
$\mathbb{Q}$, there is nothing special about modular forms for the group
$\text{SL}_2(\mathbb{Z})$, because they do not take into account the fact that the set of algebraic integers of $F$ is larger than that of $\mathbb{Q}$. Rather, we want an automorphic form
for the \emph{Hilbert modular group}, the group
$\text{SL}_2(\mathcal{O}_K)$, which consists of two-by-two matrices of
determinant one with entries in $\mathcal{O}_K$, the ring of integers of
the real subfield of $F$. For an introduction to the theory of Hilbert modular forms, see \cite{Hilbgenref}. Since there are $k$ ways to embed $K$ into
$\mathbb{R}$, there are $k$ corresponding ways to embed
$\text{SL}_2(\mathcal{O}_K)$ into $\text{SL}_2(\mathbb{R})$. We call these $k$ embeddings $\sigma_j$. We can
now define an action of the Hilbert modular group on $k$-tuples of
complex numbers $( z_1, z_2, \dots, z_k ) \in \mathcal{H}^k$ by

\begin{equation*}
\gamma ( z_1, z_2, \dots, z_k ) = ( \sigma_1(\gamma) z_1,
\sigma_2(\gamma) z_2, \dots, \sigma_k(\gamma) z_k ),\ \gamma \in \text{SL}_2(\mathcal{O}_K).
\end{equation*}

\noindent We want to find an automorphic form for this action of the Hilbert
modular group that is characteristic of the field $F$. We define the
\emph{$\psi$ function} as

\begin{equation*}
  \psi_F(z_1, z_2, \dots, z_k)=\sum\limits_{\alpha \in \mathcal{O}_F} e^{\pi i \langle
\alpha, \alpha \rangle_{z_1,z_2,\dots,z_k}}. 
\end{equation*}

\noindent This function is examined in \cite{Garrett}. When $F$ is a quadratic
field, the $\psi$ function is simply a modular form, since
there is only a single embedding of $F$ into $\mathbb{C}$ up to
conjugation; these modular forms have been studied by Hecke \cite{Hecke} and others. For a general CM field, the $\psi$ function is automorphic for a subgroup of the Hilbert modular group
\cite{Hecke}.

When all the arguments are identical, the $\psi$ function of the field $F$ reduces to the $\theta$ function of $\mathcal{O}_F$,

\begin{equation*}
\psi_F(z,z,\dots, z) = \theta_{\mathcal{O}_F}(z).
\end{equation*}

\noindent More interestingly, the $\psi$ function contains the $\theta$ function of any
principal ideal of $\mathcal{O}_F$. If $\mathcal{I}$ is a principal
ideal of $\mathcal{O}_F$ generated by an element $\kappa$, we have

\begin{equation*}
\begin{array}{rl}
\psi_F(z \sigma_1(\kappa \bar{\kappa}), z \sigma_2(\kappa \bar{\kappa}), \dots, z
\sigma_k(\kappa \bar{\kappa})) & = \sum\limits_{\alpha \in \mathcal{O}_F}
e^{\pi i \langle
\alpha, \alpha \rangle_{z \sigma_1(\kappa \bar{\kappa}), z \sigma_2(\kappa \bar{\kappa}), \dots, z
\sigma_k(\kappa \bar{\kappa})}} \\
 & = \sum\limits_{\alpha \in \mathcal{O}_F} e^{\pi i
   \langle \kappa \alpha, \kappa \alpha \rangle_{z,z,\dots,z}} \\
 & = \sum\limits_{\alpha \in \mathcal{I}} e^{\pi i
   \langle \alpha, \alpha \rangle_{z,z,\dots,z}} \\
 & = \theta_\mathcal{I}(z).
\end{array}
\end{equation*}

\noindent Thus, we can see that the $\psi$ function encodes much richer information about
$\mathcal{O}_F$ than the $\theta$ function does by itself.

Finally,
we consider what happens as we let all of the arguments $z_j$ of
$\psi_F$ go to $i \infty$, i.e., how does $\psi_F$ behave at the
cusp? Unlike ordinary modular forms, there are many ways the arguments can
approach $i \infty$. We can assume that all of the $z_j$ are purely
imaginary, as other situations can be covered by analytic
continuation. Further, we can chose a path by fixing the ratios of the $z_j$, and introducing a real parameter $t$, which we will let go to infinity. Combining these, we write $z_j = t i x_j$, where $x_j$ are real weights. We then have

\begin{equation*}
\begin{array}{rl}
\psi_F(t i x_1, t i x_2, \dots, t i x_k) & = \sum\limits_{\alpha \in
  \mathcal{O}_F} e^{- t \pi \langle \alpha, \alpha \rangle_{x_1, x_2,
    \dots, x_k}} \\
 & = 1 + n e^{- t \pi \mu} + \text{o} (e^{- t \pi \mu}),
\end{array}
\end{equation*}

\noindent as $t$ goes to $\infty$, where $\mu$ the minimal value of the weighted norm for the given ratios of $x_j$ and $n$ is the number of
vectors with norm $\mu$. This means that the behavior of $\psi_F$ at
the cusp is determined by those vectors in $\mathcal{O}_F$ which have
minimum weighted norms for some choice of $x_j$. Because the $\psi$
function contains the $\theta$ functions of all principal ideals of
$\mathcal{O}_F$, understanding the $\psi$ function's behavior at the cusp
will tell us about the behavior of each $\theta$ function at the cusp,
which is equivalent to knowing the shortest vectors in each principal ideal. The fact that the $\psi$ function can contain information about all principal ideals of the number field is a direct consequence of its incorporation of the more flexible weighted norm discussed above, as opposed to the field trace.

\section{Preliminaries}
\label{prelims}

In this section, we formally define the objects we use in the proof of our main result.

\subsection{Lattices from CM number fields}

In this section, we show how to generate a lattice from a CM number field.

\begin{defn}
\label{CM}
A number field $F$ is a \emph{CM number field} if it is a totally imaginary quadratic
extension of a totally real field $K$.
\end{defn}

In Definition \ref{CM}, by \emph{totally real field}, we mean a number
field for which every complex embedding lies in the real numbers, and by \emph{totally imaginary}, we mean a field that cannot be embedded into $\mathbb{R}$. An
immediate consequence of this definition is that the field $F$ has a
complex conjugation operation which is independent of its embedding into
$\mathbb{C}$. To see this, note that the Galois group $\text{Gal}(F/K)
\cong \mathbb{Z}/2\mathbb{Z}$
is the subgroup of $\text{Gal}(F/\mathbb{Q})$ which fixes $K$. Since any
embedding $\sigma$ of $K$ into $\mathbb{R}$ can be identified with an element of
the group $\text{Gal}(K/\mathbb{Q})$, we get a pair of conjugate
embeddings $\sigma, \bar{\sigma}$ of $F$ into $\mathbb{C}$, which are
identified with the coset $\sigma \cdot \text{Gal}(F/K)$ in $\text{Gal}(F/\mathbb{Q})$.

As in Section \ref{motiv}, we denote the set of $k$ embeddings of $K$ into $\mathbb{R}$ by $\{
\sigma_1, \sigma_2, \dots, \sigma_k \}$. Each embedding $\sigma_j$ gives a norm,
which sends $\alpha \mapsto \sigma_j(\alpha\bar{\alpha})$. We combine these into
a single map, the weighted norm, which depends on a set of $k$ weights
$\{ x_1, x_2, \dots, x_k \}$. To prove that this is a norm, we note that
it is sufficient to show that each map $\alpha \mapsto
\sigma_j(\alpha\bar{\alpha})$ is a norm. This is true because, in any embedding,
mapping a complex number $z \mapsto z\bar{z}$ is a norm on
$\mathbb{C}$, and consequently is also a norm on $\sigma_j(F)$.

Equipped with the weighted norm, the ring of integers $\mathcal{O}_F$ of the number field $F$ becomes a lattice, which is the primary lattice we mean when we refer to lattices generated from CM number fields. Additionally, each principal ideal of $\mathcal{O}_F$ also becomes a lattice when equipped with the weighted norm.

\subsection{Group of units}

We define notation, related to the group of units of $\mathcal{O}_F$, which we use in our proof. By Dirichlet's unit theorem (see \cite{algnum}), we know that the group of units, modulo torsion, in
$\mathcal{O}_F$ is free abelian on $(k-1)$ generators; this is equivalent to the group of units, modulo $\{\pm 1\}$, of $\mathcal{O}_K$, the ring of integers of the real field $K$, because $F$ is a CM field. We
denote this free abelian group by $U$, and a set of generators of $U$ by $\{ g_1, g_2, \dots, g_{k-1} \}$. We define sets of $k$ points $\Delta_s$ indexed by $s \in S_{k-1}$, where $S_{k-1}$ is the group of permutations of $\{ 1,2, \dots, k-1 \}$, by

\begin{equation*}
\Delta_s = \{1, g_{s(1)}, g_{s(1)}g_{s(2)}, \dots, \prod_{l=1}^{k-1} g_{s(l)}\}.
\end{equation*}

We call the union of the $\Delta_s$ the \emph{fundamental domain}, which we denote $\mathcal{D}$. If we consider $U \cong \mathbb{Z}^{k-1}$ to be a subset of $\mathbb{R}^{k-1}$, then each $\Delta_s$ forms the vertices of a $(k-1)$-simplex, and the fundamental domain forms the vertices of the standard $(k-1)$-dimensional unit hypercube:

\begin{equation*}
\bigcup_{s \in S_{k-1}} \Delta_s = \left\{ \prod_{j = 1}^{k-1} g_j^{a_j}\ |\ a_j \in \{ 0,1\} \right\}.
\end{equation*}

\subsection{Additional notation}

We introduce some additional notation that will be useful in stating our main
theorem. The field norm is the map from $F$ to $\mathbb{Q}$ that takes $\alpha
\mapsto N_\mathbb{Q}(\alpha) = \sigma_1(\alpha\bar{\alpha}) \sigma_2(\alpha\bar{\alpha}) \dots
\sigma_k(\alpha\bar{\alpha})$. We also define a map from $F$ to
$\mathbb{R}^k$, denoted $\Sigma$, that maps

\begin{equation*}
\Sigma: \alpha \mapsto \left(
  \sigma_1(\alpha\bar{\alpha}),\sigma_2(\alpha\bar{\alpha}), \dots, \sigma_k(\alpha\bar{\alpha}) \right).
\end{equation*}

\noindent We also denote by $\Sigma(\Delta_s)$ the simplex with vertices at $\{ \Sigma(u_0), \Sigma(u_1), \dots, \Sigma(u_{k-1}) \}$, where $\Delta_s = \{ u_0, u_1, \dots, u_{k-1} \}$.

\section{Statement of bound}
\label{statement}

The main theorem of this paper establishes a bound on the field norm of any vector which has minimal weighted norm for some
choice of weights. We first define matrices which are used in the statement of the theorem.

\begin{defn}
Let the set of vertices of a $k$-simplex in $\mathbb{R}^k$ be denoted
$\{ v_1, v_2, \dots, v_k \}$. We define the matrix $A$ to be the $k
\times k$ matrix whose $j$th row is the vector $v_j$. Let $B$ be the $(k-1) \times k$ matrix whose $j$th row
is $(v_{j+1}-v_j)$. The matrix $B_l$ is defined as the $(k-1) \times (k-1)$ matrix formed from
$B$ by deleting the $l$th column.
\end{defn}

We can now state our main theorem.

\begin{thm}
\label{bound}
Let $F$ be a CM number field, and let $\mathcal{O}_F$ be its ring of
integers. Let $\alpha$ be a nonzero vector which has a minimal weighted norm
$\langle \alpha, \alpha \rangle_{x_1, x_2, \dots, x_k}$ for some choice
of $\{ x_1, x_2, \dots, x_k \}$. Then, $\alpha$ must satisfy:
\begin{equation*}
N_\mathbb{Q}(\alpha) \leq \text{\emph{max}} \left|  \left(
    \frac{\text{\emph{det}} A}{k} \right)^k
  \prod_{l = 1}^k \frac{1}{\text{\emph{det}} B_l} \right|,
\end{equation*}
\noindent where the maximum is taken over the simplices $\Sigma(\Delta_s)$ for all
$s \in S_{k-1}$.
\end{thm}

We prove Theorem \ref{bound} in Section \ref{proof}. 

\section{Proof of bound (Theorem \ref{bound})}
\label{proof}

\subsection{Reduction to convex hull}

We want to find the set of points
$\Sigma(\alpha) \in M$ such that $\alpha$ is a minimum for some weighted norm over
$\mathcal{O}_F$, where $M$ the nonzero points in the image of the map $\Sigma:
\mathcal{O}_F \rightarrow \mathbb{R}^k$. We have the following lemma.

\begin{lem}
\label{hull}
An element $\alpha \in \mathcal{O}_F$ has a minimum weighted norm for
some choice of weights if and only if $\Sigma(\alpha)$ is on the
boundary of the convex hull of $M$.
\end{lem}

\begin{proof}

For given $\{
x_1,x_2,\dots,x_k \}$ and
$c$, we define a hyperplane in $\mathbb{R}^k$ by

\begin{equation*}
\sum_{j=1}^k x_j \xi_j = c,
\end{equation*}

\noindent where $\xi_j$ are the standard coordinates on $\mathbb{R}^n$. If we have a point
$\Sigma(\alpha)$ that is a minimum point for some choice of $\{ x_1,x_2,\dots,x_k \}$, then every point in $M$ will have 

\begin{equation*}
\sum_{j=1}^k x_j \xi_j \geq \langle \alpha, \alpha \rangle_{x_1, x_2, \dots, x_k}.
\end{equation*}

\noindent This means that every point in $M$ will lie on or above the hyperplane
defined by $\{
x_1,x_2,\dots,x_k \}$ and $\langle \alpha, \alpha \rangle_{x_1, x_2,
  \dots, x_k}$. As we vary the weights, these hyperplanes trace out the boundary of the convex hull
of the set $M$. Also, for any point $\Sigma(\alpha)$ in $M$ on the convex hull of $M$,
there is some choice of weights that makes the weighted norm of $\alpha$ a minimum,
by the definition of the convex hull.

\end{proof}

\subsection{Tiling with units}
\label{tiling}

Now that we have shown in Lemma \ref{hull} that the minimal vectors occur on the boundary convex hull of $M$, we want to use this fact to place an upper bound on the field norm of a minimal vector. Consider the subset $\Sigma(U)$ of $M$, where $U$ is the group of units of $\mathcal{O}_F$. It is a basic result from number
theory (see \cite{algnum}) that the field norm maps algebraic integers into
$\mathbb{Z}$. From this, we can conclude that $N_\mathbb{Q}(u) = 1$ for
a unit $u$, since 1 is the only positive unit in the
integers. Conversely, if an element $u \in \mathcal{O}_F$ has
$N_\mathbb{Q}(u) = 1$, we can conclude that $u$ is a unit, since its
inverse is the product of its Galois conjugates. Thus, the set
$\Sigma(U)$ is precisely the set of points in $M$ that lie on the
hyperboloid $\xi_1 \xi_2 \dots \xi_k = 1$.

We want to use the set $\Sigma(U)$ to bound the field norm of
a point on the boundary of the convex hull of $M$. Consider a
\emph{triangulation} of $\Sigma(U)$, that is, a way of attaching $(k-1)$-simplices to $\Sigma(U)$ so that the surface $T$ formed by their union
divides the space $\mathbb{R}^k$ into two disjoint pieces. We obtain the
following result.

\begin{lem}
\label{triang}
Let $\alpha$ be an element of $\mathcal{O}_F$. Given a triangulation of
$\Sigma(U)$, if $\alpha$ has a minimum weighted norm for
some choice of weights, then $N_\mathbb{Q}(\alpha) \leq
\text{\emph{max}}\  N_\mathbb{Q}(\beta)$, where the maximum is taken over
points $\beta$ in the surface $T$.
\end{lem}

\begin{proof}

Clearly, if a
point in $M$ lies above such a triangulation, it cannot be on the
boundary of the convex hull, since every simplex in the triangulation lies on or above the boundary of the convex hull. Furthermore, if a point in $M$ has field norm greater than the maximum value of $N_\mathbb{Q}$ over the
triangulation, it must lie above the triangulation and thus cannot lie on the
boundary of the convex hull. Together with Lemma \ref{hull}, this proves
Lemma \ref{triang}.

\end{proof}

We now use the fundamental domain $\mathcal{D}$ of the group of units to create a
canonical triangulation. Consider the sets $u \Delta_s$, where $u$ is any unit. Each set $u \Delta_s$, under the map $\Sigma$, forms the vertices of a simplex. Taken together, the simplices formed by the images of $u \Delta_s$ for all $u \in U$ and $s \in S_{k-1}$ give a triangulation of $\Sigma(U)$.

\subsection{Our bound}

In this section, we show that we only need to consider the simplices $\Sigma(\Delta_s)$ by proving Lemma \ref{fundamental}, and conclude the proof of Theorem \ref{bound} by calculating the maximum value of the field norm over each simplex.

\begin{lem}
\label{fundamental}
Let $\alpha$ be an element of $\mathcal{O}_F$. If $\alpha$ has a minimum weighted norm for
some choice of weights, then $N_\mathbb{Q}(\alpha) \leq
\text{\emph{max}}\  N_\mathbb{Q}(\beta)$, where the maximum is taken over
points $\beta$ on the simplices $\Sigma(\Delta_s)$.
\end{lem}

\begin{proof}
The map that allowed us to create a triangulation of $U$ from its fundamental domain as in Section \ref{tiling},
multiplication by a unit $u$, descends to a linear map on
the image $\Sigma(U)$, which maps $\xi_j \mapsto \sigma_j(u \bar{u}) \xi_j$. The hyperboloids $\xi_1\xi_2\dots \xi_k = c$ are
invariant under this map, because $\sigma_1(u \bar{u})\sigma_2(u \bar{u})\dots\sigma_k(u \bar{u}) = N_\mathbb{Q}(u) = 1$, since $u$ is a unit. Thus, the field norm is preserved by multiplication by a unit. Thus, any two
simplices that can be connected by multiplication by a unit have the
same value of $\text{max}\ N_\mathbb{Q}(\beta)$, where the maximum is
taken over each simplex independently. Since any simplex in the
canonical triangulation can be mapped to one of the simplices $\Sigma(\Delta_s)$
by multiplication by a unit, we can get our bound simply by considering
the simplices $\Sigma(\Delta_s)$.
\end{proof}

We now calculate the value of $\text{max}\ N_\mathbb{Q}(\beta)$
over a single simplex. Let the vertices of the simplex be $\{ v_1, v_2,
\dots, v_k\}$. We map the simplex with a linear
transformation, under which the hyperboloids $\xi_1\xi_2\dots \xi_k = c$ are invarient, to a simplex where the sum $c$ of each
vertex's coordinates is the same for all vertices. In this case, the maximum field norm is just $(c/k)^k$, the norm of the
centroid of the transformed simplex. 

Let $w$ be the wedge product of the $(k-1)$ vectors $(v_j-v_{j-1})$,
divided by the $k$th root of the product of the wedge product's
components. Then, if we let $w_j$ denote the $j$th component of
$\star \ w$, we have that $w_1 w_2 \dots w_k=1$, where $\star$ is the Hodge star operator. Let $A$ be the linear map $v \mapsto \text{diag}(w_1, w_2, \dots,
w_k) v$, which leaves the hyperboloids $\xi_1\xi_2\dots \xi_k = c$ invarient. The sum of the components of the image of a vector $v$ under $A$ is just $| \star (w \wedge v)\ |$. Since $w
\wedge (v_j - v_{j-1}) = 0$, we have by induction that $w \wedge v_j = w
\wedge v_1$ for all $j$. Thus, our linear map is of the desired form,
and the maximum value of $N_\mathbb{Q}(\beta)$ on the simplex is
$|\star(w \wedge v_1)/k\ |^k$.

To write this bound explicitly in terms of the vectors $v_j$, we
consider the definition of $w$. The $l$th component of $w$ is the volume
of the $(k-1)$-dimensional simplex formed by the projections of
$(v_j-v_{j-1})$ onto the hyperplane orthogonal to the $l$th basis
vector. This is exactly $\text{det} B_l$. Additionally, the product $\star \left( v_1 \wedge
(v_2 - v_1) \wedge \dots \wedge (v_k - v_{k-1}) \right) = \star \left( v_1 \wedge v_2 \wedge
\dots \wedge v_k \right) = \text{det} A$. Here, $A$ and $B_l$ are defined
as in Section \ref{statement}.

If we take these calculations and use them to rewrite our bound, we
obtain

\begin{equation*}
N_\mathbb{Q}(\beta) \leq \left|  \left(
    \frac{\text{det} A}{k} \right)^k
  \prod_{l = 1}^k \frac{1}{\text{det} B_l} \right|,
\end{equation*}

\noindent for $\beta$ in our given simplex. Taking the maximum of this bound over
$\Sigma(\Delta_s)$ for $s \in S_{k-1}$ proves Theorem \ref{bound}.

\section{Consequences of bound}
\label{consequences}

In this section, we examine two corollaries of our main theorem, Theorem \ref{bound}. The first extends Theorem \ref{bound} to any principal ideal of $\mathcal{O}_F$, and the second shows that Theorem \ref{bound} implies that there is a finite set of points in $\mathcal{O}_F$ that allows one to determine the set of minimal vectors in all principal ideals.

\subsection{Principal ideals}

Corollary \ref{principal} extends Theorem \ref{bound} to a principal ideal of $\mathcal{O}_F$.

\begin{cor}
\label{principal}
Let $\mathcal{I}$ be a principal ideal of $\mathcal{O}_F$ generated by
an element $\kappa$, and let $\alpha$ be a vector in $\mathcal{I}$ which has a minimum weighted
norm for vectors in $\mathcal{I}$. Then, $\alpha$ satisfies:
\begin{equation*}
N_\mathbb{Q}(\alpha) \leq \text{\emph{max}} \left|  \left(
    \frac{\text{\emph{det}} A}{k} \right)^k
  \prod_{l = 1}^k \frac{1}{\text{\emph{det}} B_l} \right| N_\mathbb{Q}(\kappa),
\end{equation*}
\noindent where the maximum is taken over the simplices $\Sigma(\Delta_s)$.
\end{cor}

\begin{proof}
Since $\mathcal{I}$ is principal, every element $\alpha$ can be written
as $\alpha' \kappa$ for $\alpha' \in \mathcal{O}_F$; such an $\alpha$ is
minimal over $\mathcal{I}$ if and only if $\alpha'$ is minimal over
$\mathcal{O}_F$. In this case, we can apply Theorem \ref{bound} to the
expression $N_\mathbb{Q}(\alpha) = N_\mathbb{Q}(\alpha' \kappa) =
N_\mathbb{Q}(\alpha') N_\mathbb{Q}(\kappa)$, and Corollary
\ref{principal} follows.
\end{proof}

\subsection{Existence of characteristic finite set}

Corollary \ref{finite} proves the existence of a finite set of points in $\mathcal{O}_F$ that contain all the information about minimal vectors in any principal ideal of $\mathcal{O}_F$.

\begin{cor}
\label{finite}
There exists a finite set $E$ of elements of $\mathcal{O}_F$ such that if $\alpha$ is a minimal vector for some choice of weighted norm in $\mathcal{O}_F$ or in a principal ideal $\mathcal{I}$ of $\mathcal{O}_F$, then $\alpha$ can be transformed canonically into a corresponding point $\eta \in E$. Further, the choice of weighted norm for which $\alpha$ is minimal is transformed into a weighted norm for which the corresponding point $\eta$ in $E$ is minimal over $\mathcal{O}_F$.
\end{cor}

\begin{proof}

The only points in $\mathcal{O}_F$ which could possibly be minimal for some weighted norm have images under $\Sigma$ that lie between the hyperboloid $\xi_1\xi_2\dots \xi_n = 1$ and the
triangulation of $\Sigma(U)$. By the proof of Lemma \ref{fundamental}, this
region can be split up into chambers, each of which can be transformed
into the chamber lying underneath the fundamental domain by a linear
transformation $T$ corresponding to multiplication by a unit. We denote by $E$ the set of point in $\mathcal{O}_F$ with images inside this fundamental chamber. Since these transformations leave $\mathcal{O}_F$ invariant, they take any minimal point $\alpha \in \mathcal{O}_F$ to a minimal point $\eta \in E$. Additionally, we obtain a corresponding choice of weighted norm for which $\eta$ is minimal by pre-composing the weighted norm for which $\alpha$ is minimal with the inverse of the transformation $T$ which takes $\alpha$ to $\eta$, i.e., we take the pullback of the weighted norm along $T^{-1}$.

For a principal ideal $\mathcal{I}$, as the proof of Corollary
\ref{principal} and the discussion in Section \ref{motiv} indicate, there is a linear transformation $S$ that takes the elements of $\mathcal{I}$ to the elements of $\mathcal{O}_F$, which is the inverse of multiplication by the generator of the ideal. The map $S$ takes a minimal element $\beta \in \mathcal{I}$ for one choice of weighted norm to a minimal point in $\mathcal{O}_F$ for a different weighted norm, the pullback of the weighted norm along $S^{-1}$. This can then be transformed into an element $\eta \in E$ and a corresponding weighted norm as above.

All that remains to prove the first part of Corollary \ref{finite} is to show that $E$ is finite. Since the fundamental chamber is compact, and
$\Sigma(\mathcal{O}_F)$ is discrete, the image $\Sigma(E)$ must be finite. We recall that $\Sigma(\mu) = (\sigma_1(\mu \bar{\mu}),\dots, \sigma_k(\mu \bar{\mu}))$. Since each embedding $\sigma_j$ is injective, if $\Sigma(\mu) = \Sigma(\nu)$, then $\mu \bar{\mu} = \nu \bar{\nu}$. The kernel of the map $\mu \mapsto \mu \bar{\mu}$ from $F$ to $K$ consists of the torsion elements of the group of units of $\mathcal{O}_F$, by Dirichlet's unit theorem (see \cite{algnum}). This is finite, because $F$, as a number field, is a finite degree extension of the rationals. Thus, the preimage in $\mathcal{O}_F$ of a point in $\Sigma(E)$ is finite, so the set $E$ itself must be finite.
\nopagebreak
\end{proof}

\begin{rem}
The proof of Corollary \ref{finite} relies heavily on the fact that we can vary the weights that define the weighted
norm. Many times throughout the proof, we exchange one choice of weighted norm for another, so that a minimal vector under the first choice gets sent to a minimal vector under the second choice. We can do this because the pullback of a weighted norm along a map corresponding to multiplication by an element of $F$ is always a weighted norm for some other choice of weights. The important thing to note about Corollary \ref{finite} is that once the set $E$ is completely understood, we can find the minimal vectors for any given choice of weighted norm by running Corollary \ref{finite} backwards. We first identify the correct linear transformation to turn a known choice of the weighted norm into an unknown choice, and use that transformation to map a minimal point in $E$ to a new, unknown minimal point.
\end{rem}

\section{Craig's Lattices}
\label{craig's}

We conclude by considering some applications of Theorem \ref{bound} to
the Craig's Difference Lattice Problem. In this case, $F =
\mathbb{Q}[\zeta_n]$ and $\mathcal{O}_F = \mathbb{Z}[\zeta_n]$. Our
principal ideals are generated by $(1-\zeta_n)^r$. We want to completely
classify the shortest vectors for the specific choice of equal weights
to define the norm.

In addition to Theorem \ref{bound}, we have the result from class field
theory (see \cite{classfield}) that $N_\mathbb{Q}(\alpha)$ can only be one or something $\geq n$ for nonzero $\alpha
\in \mathbb{Z}[\zeta_n]$. We use this alongside Theorem \ref{bound} to
solve the cases $n=5, 7$. While these are not new results, studying how our theorem
is applied to easy cases allows us to understand what an application of
our theorem to a more general case would look like.

Let $n = 5$, and let $F$ be $\mathbb{Q}[\zeta_5]$. In this case there
are two conjugate pairs of complex embeddings, so $k=2$. We claim that the set
of minimal vectors is exactly the units of $\mathcal{O}_F$. Since the
only points with $N_\mathbb{Q}$ greater than one have $N_\mathbb{Q}$ at
least five, if we can show that our bound is less than five then only the
units can possibly be minimal. Furthermore, since the hyperboloid $\xi_1
\xi_2 = 1$ is convex, each unit can actually be obtained as a minimum.

To check if our bound is less than five, we have to check $(k-1)! =
1$ simplex. This simplex has vertices $\{(1,1), ((3+\sqrt{5})/2,
(3-\sqrt{5})/2)\}$, the images of the identity and the generator of the
group of units of $\mathcal{O}_F$. We find that $
|\text{det}A |=
\left| \text{det} \left( \begin{smallmatrix}
1 & \frac{3+\sqrt{5}}{2} \\
1 & \frac{3-\sqrt{5}}{2} \\
\end{smallmatrix} \right) \right| = \sqrt{5}.$ We also have that
$\text{det} B_1 = (1+\sqrt{5})/2$ and $\text{det} B_2 =
(1-\sqrt{5})/2$. Theorem \ref{bound} then says that, if $\alpha$ is a
minimal vector in $\mathcal{O}_F$,

\begin{equation*}
N_\mathbb{Q}(\alpha) \leq \left| \left( \frac{\sqrt{5}}{2} \right)^2
  \frac{1}{(\frac{1+\sqrt{5}}{2})(\frac{1-\sqrt{5}}{2})} \right| =
\frac{5}{4} < 5.
\end{equation*}

\noindent Thus, the minimal vectors in the Craig's lattices $\mathbb{A}_4^{(r)}$ are
of the form $(1-\zeta_5)^r u$, where $u$ is a unit.

We can apply similar arguments to the case $n=7$. In this case, $k=3$,
and we have to check $(k-1)! = 2$ simplices. Our bound gives values of
$49/27$ for one simplex and $56/27$ for the other. Theorem \ref{bound}
then tells us that if $\alpha$ is minimal,

\begin{equation*}
N_\mathbb{Q}(\alpha) \leq \frac{56}{27} < 7,
\end{equation*}

\noindent and the minimal vectors in the lattices $\mathbb{A}_6^{(r)}$ are units
scaled by $(1-\zeta_7)^r$.

This pattern does not continue for all $n$; in fact, for $n = 11$, our
bound is greater than 11, so we can not immediately determine all of the
minimal vectors. However, the methods used to prove Theorem \ref{bound}
still apply in this situation and give insight into the Craig's
Difference Lattice problem for all $n$ and $r$.

\section{Acknowledgments} 

I'd like to thank my mentor, Dmitry Vaintrob, who is currently a
graduate student at MIT. I'd also like to thank Professor Abhinav Kumar
for suggesting the problem. Thanks also goes to Tanya Khovanova and the MIT math department
for their support. I'd like to thank Professor John Rickert and Professor Jacob Sturm, as well as Simanta Gautam, Joshua
Brakensiek, Sitan Chen, and Alec Lai for their help in editing this
paper. This work would not have been possible without the
Center for Excellence in Education and the Research Science Institute,
along with my sponsors, Mr. Zachary Lemnios, Dr. Laura Adolfie, Dr. John
Fischer, and Dr. Robin Staffin from the Department of Defense, Ms. Debra
L. Waggoner from the Corning Incorporated Foundation, Dr. Xiang Gao, and
Dr. Li Tong.

\newpage

\begin{singlespace}
\thispagestyle{empty}
\bibliographystyle{style}

\bibliography{biblio}
\end{singlespace}

\end{document}